\documentclass[12pt]{amsart}

\usepackage{amssymb}
\usepackage{verbatim}
\usepackage[toc,page]{appendix}
\usepackage{mathrsfs}

\newtheorem{thm}{Theorem}[section]

\newtheorem{lem}[thm]{Lemma}
\newtheorem{cor}[thm]{Corollary}

\theoremstyle{definition}

\theoremstyle{remark}

\numberwithin{equation}{section}

\begin{document}

\title[A note on the binary additive divisor problem]{A note on the  binary additive divisor problem}

\author{Olga  Balkanova}

\author{Dmitry  Frolenkov}
\begin{abstract}
In this note  we show that the methods of Motohashi and Meurman yield the same upper bound on the error term in the binary additive divisor problem. With this goal, we improve an estimate in the proof of Motohashi.
\end{abstract}


\maketitle

\tableofcontents


\section{Introduction}
Let $d(n)=\sum_{a|n}1$. The additive divisor problem is a statement of the form
\begin{equation}\label{eq:divisor sum}
\sum_{n=1}^{M}d(n)d(n+f)=MT(M,f)+E(M,f),
\end{equation}
where  $MT(M,f)$ is the main term defined by \cite[Eq.~1]{Me} and $E(M,f)$ is an error term.
The problem was studied in \cite{ DI, E, HB, I, K, Me, Mo}.

In this note we are mainly interested in results of Motohashi \cite{Mo} and Meurman \cite{Me}.
To estimate $E(M,f)$, they used two different methods.
Motohashi followed Kuznetsov's approach \cite{K} based on the method of analytic continuation and on  spectral decomposition of the left hand side of \eqref{eq:divisor sum} via the Kuznetsov trace formula.
Let $\alpha$ be defined by \cite[Eq.~2.17]{Mo}. For any $\epsilon>0$, Motohashi \cite[Eq.~2.18]{Mo} showed that
\begin{equation}\label{eq:Mot}
E(M,f)\ll (M^2+Mf)^{1/3+\epsilon}+f^{1/8+\alpha/2}(M^2+Mf)^{1/4+\epsilon}+f^{1/2+\alpha}M^{\epsilon}
\end{equation}
uniformly for $1\leq f \leq M^{2/(1+2\alpha)}$.

Meurman's proof is completely different; it makes use of Heath-Brown's representation for the divisor function \cite{HB}, the Kuznetsov trace formula and large sieve inequalities. As a result, Meurman \cite[Eq.~4]{Me} proved that for any $\epsilon>0$  one has
\begin{multline}\label{eq:Meur}
E(M,f)\ll (M^2+Mf)^{1/3}M^{\epsilon}+\\(M^2+Mf)^{1/4}M^{\epsilon}\min\left(M^{1/4}, f^{1/8+\alpha/2}\right)
\end{multline}
uniformly for $1\leq f \leq M^{2-\epsilon}$.

Note that in the range $1\leq f\leq M^{2/(1+4\alpha)}$ estimates \eqref{eq:Mot} and \eqref{eq:Meur} are the same. However, when $f\geq M^{2/(1+4\alpha)}$ estimate \eqref{eq:Meur} is better than \eqref{eq:Mot}.

The two considered methods can be applied to study a wide variety of problems in analytic number theory. Therefore, it is important to understand if one approach is better than another.

In case of the binary additive divisor problem, we show that the methods of Motohashi and Meurman are equally good and yield the same estimates in all ranges. With this goal, we analyze the proof of Motohashi and identify steps that can be taken to sharpen \eqref{eq:Mot}.

\section{Analysis of  Motohashi's proof}
We keep notations of \cite{Mo}.
According to \cite[Eq.~5.3~and~Theorem~3]{Mo} the error term in the additive divisor problem can be decomposed into continuous,  discrete and holomorphic spectra
\begin{equation*}
E^*(M;f)-E^*(M/2;f)=e_1(M;f)+e_2(M;f)+e_3(M;f)+O(\delta M^{1+\epsilon}),
\end{equation*}
where
\begin{equation*}
e_1(M;f)=\frac{\sqrt{f}}{\pi}\int_{-\infty}^{\infty}\frac{f^{-ir}\sigma_{2ir}(f)|\zeta(1/2+ir)|^4}{|\zeta(1+2ir)|^2}\Theta(r;W)dr,
\end{equation*}
\begin{equation*}
e_2(M;f)=\sqrt{f}\sum_{j=1}^{\infty}\alpha_jt_j(f)H_{j}^{2}(1/2)\Theta(\kappa_j;W),
\end{equation*}
\begin{equation*}
e_3(M,f)=\frac{\sqrt{f}}{4}\sum_{k=6}^{\infty}\sum_{j \leq \theta(k)}(-1)^k\alpha_{j,k}t_{j,k}(f)H_{j,k}^2(1/2)\Xi_0(k-1/2;W).
\end{equation*}

It was proved in \cite[Eq.~5.18,~5.19]{Mo}  that for any $\epsilon>0$  and $M\leq f<M^{2-\epsilon}$ one has
\begin{equation*}
e_1(M;f), e_3(M;f)\ll f^{1/2+\epsilon}.
\end{equation*}
The largest contribution to the error term comes from the discrete spectrum $e_2(M;f)$. 

The main observation is that the estimate of the first summand in \cite[Eq.~5.20]{Mo} is not sharp. In order to obtain an improvement it is sufficient to optimize \cite[Eq.~5.10]{Mo} by proving a better estimate for the following integral
\begin{equation}\label{eq:lambda}
\Lambda(r,Z)=\frac{\Gamma^2(1/2+ir)}{\Gamma(1+2ir)}\int_{0}^{\infty}\frac{g\left( a/y\right)}{y^{3/2-ir}}
F\left( \frac{1}{2}+ir,\frac{1}{2}+ir,1+2ir;-y\right)dy,
\end{equation}
where $a=1/Z=f/M$. Note that $g(x)$ is a smooth characteristic function of the interval $[1/2;1]$, see  \cite[p.~561]{Mo} for details.  Furthermore, \cite[Theorem~ 3]{Mo} provides the  relation between two functions  
\begin{equation}\label{theta=lambda}
\Theta(r;W)=\frac{1}{2}\Re\left(\left(1+\frac{i}{\sinh(\pi r)}\right)\Lambda(r,Z)\right).
\end{equation}

\section{The main estimate}
The main result of this section is the upper bound for $\Lambda(r,Z)$ .
\begin{lem}\label{lem:mainest}
One has
\begin{equation}\label{eq:Mot2}
|\Lambda(r,Z)|\ll Z(\log{r}+|\log{Z}|)+Z^{3/2}r.
\end{equation}
\end{lem}

 The proof of Lemma \ref{lem:mainest} is based on the following estimate.
\begin{lem}\label{lem:int}
For $r\gg1$ one has
\begin{equation}
I:= \frac{1}{2\pi i}\int_{-\infty}^{+\infty}
\frac{\Gamma^2(-1/2+i(r+t))\Gamma(1-it)}{\Gamma(i(2r+t))}y^{-1+it}dt \ll \frac{r}{y}.
\end{equation}
\end{lem}
\begin{proof}
We start with estimating the integral $I$ by absolute values.
The first gamma function can be estimated using \cite[Eq.~5.5.1, ~5.4.4]{HMF} as follows
\begin{multline*}
\left|\Gamma^2(-1/2+i(r+t))\right|=\left|\frac{\Gamma^2(1/2+i(r+t))}{(-1/2+i(r+t))^2}\right|=\\
\frac{\pi}{\left|-1/2+i(r+t)^2 \right|^2\cosh{\pi(r+t)}}\ll \frac{1}{(1+(r+t)^2)\cosh{\pi(r+t)}}.
\end{multline*}
Furthermore, by \cite[Eq.~ 5.4.3]{HMF} one has
\begin{equation*}
\left|\Gamma(i(2r+t)) \right|=\sqrt{\frac{\pi}{(2r+t)\sinh{\pi(2r+t)}}}.
\end{equation*}
Finally, \cite[Eq. ~5.5.1, 5.4.3]{HMF} yield
\begin{equation*}
\left|\Gamma(1-it) \right|=|-it\Gamma(-it)|=|-it|\sqrt{\frac{\pi}{t\sinh{\pi t}}}\ll \sqrt{\frac{t}{\sinh{\pi t}}}.
\end{equation*}
To sum up, we proved that
\begin{equation*}
|I|\ll \frac{1}{y}\int_{-\infty}^{\infty}\frac{\sqrt{(2r+t)\sinh{\pi(2r+t)}}}{(1+(r+t)^2)\cosh{\pi(r+t)}}\sqrt{\frac{t}{\sinh{\pi t}}}dt.
 \end{equation*}
The last  integral can be decomposed into the sum of seven parts as follows
\begin{equation*}
\int_{-\infty}^{\infty}=\int_{1}^{\infty}+\int_{-1}^{1}+\int_{-r+1}^{-1}+\int_{-r-1}^{-r+1}+\int_{-2r+1}^{-r-1}+\int_{-2r-1}^{-2r+1}+\int_{-\infty}^{-2r-1}
\end{equation*}
so that
\begin{equation*}
|I|\ll \frac{1}{y}\sum_{i=1}^{7}I_i.
\end{equation*}
Outside the intervals $2,4,6$ the hyperbolic functions under the integral can be majorated by
\begin{equation*}
\exp\left(\frac{\pi}{2}|2r+t|-\pi|r+t|-\frac{\pi}{2}|t|\right).
\end{equation*}
Note that
\begin{equation*}
\left|r+\frac{t}{2}\right|-|r+t|-\frac{1}{2}|t|=
\begin{cases}
-t & t>0\\
0 & -r<t<0\\
2r+2t & -2r<t<-r\\
t & t<-2r
\end{cases}.
\end{equation*}
The largest contribution to $I$ comes from $I_2$, $I_3$ and $I_4$.
Estimating the integral $I_2$, we find
\begin{equation*}
I_2\ll \int_{-1}^{1}\frac{|2r+t|^{1/2}|2+t|^{1/2}}{1+(r+t)^2}dt\ll r^{-3/2}.
\end{equation*}
The integral $I_3$ is bounded by
\begin{multline*}
I_3\ll \int_{-r+1}^{-1}\frac{|2r+t|^{1/2}|t|^{1/2}}{1+(r+t)^2}dt\ll r^{1/2}\int_{1}^{r-1}\frac{t^{1/2}}{1+(r-t)^2}dt\ll\\ r^{1/2}\int_{1}^{r-1}\frac{t^{1/2}}{(r-t)^2}dt=r^{1/2}\int_{1}^{r-1}\frac{(r-y)^{1/2}}{y^2}dy\ll\\
r^{1/2}\left( \int_{1}^{r/2}\frac{r^{1/2}}{y^2}dy+\int_{r/2}^{r-1}\frac{(r-y)^{1/2}}{y^2}dy\right)\ll r.
\end{multline*}
The integral $I_4$ can be estimated as follows
\begin{equation*}
I_4\ll \int_{-r-1}^{-r+1}\frac{|2r+t|^{1/2}|t|^{1/2}}{1+(r+t)^2}dt\ll r.
\end{equation*}
Finally,
\begin{equation*}
|I|\ll \frac{1}{y}\sum_{i=1}^{7}I_i \ll \frac{1}{y}(I_2+I_3+I_4)\ll \frac{r}{y}.
\end{equation*}
\end{proof}

\begin{cor}\label{cor:hypergeometric}
One has
\begin{multline}
\frac{\Gamma^2(1/2+ir)}{\Gamma(1+2ir)}F\left( \frac{1}{2}+ir,\frac{1}{2}+ir,1+2ir;-y\right)dy=\\
y^{-1/2-ir}\left(\log{y}+
2\psi(1)-2\psi(1/2+ir)\right)+O\left(\frac{r}{y} \right),
\end{multline}
where $\psi(s)=\Gamma'(s)/\Gamma(s)$ is the logarithmic derivative of the Gamma function.
\end{cor}
\begin{proof}
The Mellin-Barnes representation yields
\begin{multline}
\frac{\Gamma^2(1/2+ir)}{\Gamma(1+2ir)}F\left( \frac{1}{2}+ir,\frac{1}{2}+ir,1+2ir;-y\right)dy=\\
\frac{1}{2\pi i}\int_{(c)}\frac{\Gamma^2(1/2+ir+s)\Gamma(-s)}{\Gamma(1+2ir+s)}y^sds,
\end{multline}
where $-1/2<c<0$.

Moving the contour of integration to $c=-1$ and letting $s:=-1+it$, one obtains
\begin{multline}
y^{-1/2-ir}\left(\log{y}+
2\psi(1)-2\psi(1/2+ir)\right)+\\\frac{1}{2\pi i}\int_{-\infty}^{+\infty}
\frac{\Gamma^2(-1/2+i(r+t))\Gamma(1-it)}{\Gamma(i(2r+t))}y^{-1+it}dt.
\end{multline}
The assertion follows by estimating the integral above using Lemma \ref{lem:int}.
\end{proof}

Since $g(x)$ is not an oscillatory function, Lemma \ref{lem:mainest} is proved by replacing everything in
\eqref{eq:lambda} by absolute values and applying Corollary \ref{cor:hypergeometric}.

Using inequality \eqref{eq:Mot2} instead of \cite[Eq.~5.10]{Mo} gives significant improvement in estimating $e_2(M,f)$  when $f>M^{2/(1+4\alpha)}$, as we now show.

\begin{thm}\label{thm:main} For any $\epsilon>0$ and $M^{2/(1+4\alpha)}<f<M^{2-\epsilon}$ one has
\begin{equation}
e_2(M,f)\ll f^{\epsilon}\left( f^{1/4}M^{1/2}+\delta^{-1/2}f^{1/2}\right).
\end{equation}
\end{thm}
\begin{proof}
First, we consider the sum over  $\kappa_j \leq \sqrt{f/M}=Z^{-1/2}$. It follows from \eqref{theta=lambda}, \eqref{eq:Mot2} that for such $\kappa_j$ one has
\begin{equation*}
\Theta(\kappa_j;W)\ll |\Lambda(\kappa_j,Z)|\ll Zf^{\epsilon}.
\end{equation*}
There are two ways to obtain an upper bound for the average
\begin{equation*}
\sum_{\kappa_j\leq \sqrt{f/M}}\alpha_jt_j(f)H^{2}_{j}(1/2).
\end{equation*}
On the one hand, using $|t_j(f)|\ll f^{\alpha}$ and \cite[Theorem 2]{Mo2} one has
\begin{equation*}
\sum_{\kappa_j\leq \sqrt{f/M}}\alpha_jt_j(f)H^{2}_{j}(1/2)\ll \frac{f^{1+\alpha+\epsilon}}{M}.
\end{equation*}
On the other hand, the Cauchy-Schwarz inequality, Kuznetsov's bound \cite{Kuz2}, \cite[Eq.~2.1]{Mo} and \cite[Theorem~ 5]{Mo2} imply that
\begin{multline*}
\sum_{\kappa_j\leq \sqrt{f/M}}\alpha_jt_j(f)H^{2}_{j}(1/2)\ll\\
\sqrt{\sum_{\kappa_j\leq \sqrt{f/M}}\alpha_jt_j(f)^2
\sum_{\kappa_j\leq \sqrt{f/M}}\alpha_jH^{4}_{j}(1/2)}\ll\\
\sqrt{\frac{f}{M}+f^{1/2+\epsilon}}\sqrt{\frac{f}{M}f^{\epsilon}}\ll \frac{f^{3/4+\epsilon}}{M^{1/2}}.
\end{multline*}
Therefore,
\begin{multline*}
A:=\sqrt{f}\sum_{\kappa_j\leq \sqrt{f/M}}\alpha_jt_j(f)H^{2}_{j}(1/2)\Theta(\kappa_j;W)\ll\\
\frac{Mf^{\epsilon}}{\sqrt{f}}\sum_{\kappa_j\leq \sqrt{f/M}}\alpha_jt_j(f)H^{2}_{j}(1/2)\ll
\frac{Mf^{\epsilon}}{\sqrt{f}}\frac{f}{M}\min\left( f^{1/4},f^{\alpha}\sqrt{f/M}\right).
\end{multline*}
Note that when $f>M^{2/(1+4\alpha)}$, one has 
\begin{equation*}\min\left( f^{1/4},f^{\alpha}\sqrt{f/M}\right)= f^{1/4}\text{
and }A\ll \sqrt{M}f^{1/4+\epsilon}.
\end{equation*}
For the sum over $\kappa_j>\sqrt{f/M}$ we apply estimates of Motohashi, proving that
\begin{multline*}
e_2(M,f)\ll M^{1/2}f^{1/4+\epsilon}+(Mf)^{1/4}\left( M^{1/4}+\delta^{-1/2}f^{1/4}M^{-1/4}\right)M^{\epsilon}
\ll\\
f^{\epsilon}\left( f^{1/4}M^{1/2}+\delta^{-1/2}f^{1/2}\right).
\end{multline*}
\end{proof}

As a consequence of Theorem \ref{thm:main} and \cite[Eq~5.3,~5.18,~5.19]{Mo}, for $M^{2/(1+4\alpha)}<f<M^{2-\epsilon}$ one has
\begin{equation*}
E^*(M)-E^*(M/2)\ll f^{1/2+\epsilon}+(\delta M+\delta^{-1/2}f^{1/2}+f^{1/4}M^{1/2})f^{\epsilon}.
\end{equation*}
By making equal the two summands, the optimal choice of parameter $\delta$ turns out to be $$\delta=\frac{f^{1/3}}{M^{2/3}}.$$
Therefore, for $M^{2/(1+4\alpha)}<f<M^{2-\epsilon}$
\begin{equation}\label{eq:E}
E^*(M)-E^*(M/2)\ll f^{\epsilon}(f^{1/3}M^{1/3}+f^{1/4}M^{1/2}).
\end{equation}
This implies that bounds \eqref{eq:E}  and \eqref{eq:Meur} are equal when $f>M^{2/(1+4\alpha)}$.

\nocite{}

\end{document}